\documentclass[copyright,creativecommons]{eptcs}
\pdfoutput=1
\providecommand{\event}{ACT 2022}
\usepackage{underscore}
\usepackage{default}
\title{Monoidal Width: Capturing Rank Width}
\def\titlerunning{Monoidal Width: Capturing Rank Width}
\author{Elena Di Lavore\institute{Tallinn University of Technology} \and Pawe{ł} {Sobociński}\institute{Tallinn University of Technology}}
\def\authorrunning{Elena Di Lavore and Pawe{ł} {Sobociński}}

\begin{document}
\maketitle

\begin{abstract}
  Monoidal width was recently introduced by the authors as a measure of the complexity of decomposing morphisms in monoidal categories. We have shown that in a monoidal category of cospans of graphs, monoidal width and its variants can be used to capture tree width, path width and branch width.
  In this paper we study monoidal width in a category of matrices, and in an extension to a \emph{different} monoidal category of open graphs, where the connectivity information is handled with matrix algebra and graphs are composed along edges instead of vertices. We show that here monoidal width captures rank width: a measure of graph complexity that has received much attention in recent years.
  %
\end{abstract}

\section{Introduction}
Many applications of category theory rely on monoidal categories as algebras of processes~\cite{fritz2020,cho2019,GhicaJL17,Comfort2021,Boisseau2021,fong2018seven,Bonchi0Z21,Coecke2017,DuncanKPW20,GhaniHWZ18}.
Morphisms are compound processes, defined as parallel and sequential compositions of simpler process components.
The compositional nature of this modelling allows a compositional computation of the underlying semantics.
But how efficient is this computation?
Given two processes \(f\) and \(g\), we can compute their semantics separately.
However, computing the semantics of their sequential composition \(f \dcomp g\) often requires an additional cost~\cite{rathke2014compositional}. Indeed, the semantics of sequential composition often means resource sharing or synchronisation along the common boundary. This in turn carries a computational burden, dependent on the size of the boundary.
On the other hand, computing the semantics of a parallel composition \(f \tensor f'\) typically does not involve any resource sharing, as indicated by the wiring of the string diagrams, and thus typically does not require significant additional computational resources.
Taking this into account, the choice of the \emph{recipe} for a morphism in terms of parallel and sequential compositions influences the cost of computing its semantics.
As shown in \Cref{fig:ex-cuts}, where vertical cuts represent sequential compositions and horizontal cuts represent parallel compositions, the same morphism can be defined in different ways with possibly different computational costs.
\begin{figure}[h!]
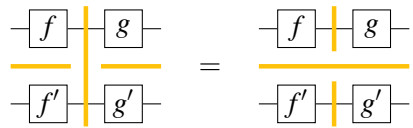

  \[\interchangelawDecFig\]
  \caption{Two monoidal decompositions of the same morphism, the right one being the cheapest.}\label{fig:ex-cuts}
\end{figure}
Given a morphism, it is thefore desirable to find the least costly recipe of \emph{decomposing} it in terms of more primitive components. We can rephrase the original question:
what is the most efficient way to decompose a morphism in a monoidal category?

The authors recently proposed \emph{monoidal width}~\cite{2022monoidalwidth} as a way of assigning a natural number to a morphism of a monoidal category, representing -- roughly speaking -- the cost of its most efficient decomposition.
In turn, this is related to the cost of computing the semantics of this morphism.

Computing efficient decompositions is not a new problem.
The graph theory literature abounds~\cite{bertele1973treewidth,halin1976treewidth,robertson1986graph-minorsII,robertson1983graph-minorsI,robertson1991graph-minorsX,oum2006rank-width,courcelle2000upper-clique,adolphson1973optimal,aharonov2006quantum,chudnovsky2011well} with notions of complexity of graphs that ultimately measure the difficulty of decomposing a graph into smaller components by cutting along the vertices or the edges of the graph.
Measures such as tree width~\cite{bertele1973treewidth,halin1976treewidth,robertson1986graph-minorsII}, path width~\cite{robertson1983graph-minorsI}, branch width~\cite{robertson1991graph-minorsX}, clique width~\cite{courcelle2000upper-clique} and rank width~\cite{oum2006rank-width} are motivated by algorithmic considerations.
Probably the best known among several results that etablish links with algorithms~\cite{bodlaender1992tourist,bodlaender2008combinatorial,courcelle1990monadic}, the following shows the importance of tree width.
\begin{theorem*}[Courcelle~\cite{courcelle1990monadic}]
  Every property expressible in the monadic second order logic of graphs can be tested in linear time on graphs with bounded tree width.
\end{theorem*}
The different notions of complexity for graphs vastly differ in low-level ``implementation details'' but they all share a similar underlying idea: that of defining decompositions and suitably measuring their efficiency.
One of our contributions is to exhibit monoidal width as a unifying framework for graph measures based on a notion of decomposition.
In fact, by choosing a suitable algebra of composition for graphs --- i.e.\ choosing the right monoidal cateory --- we recover some of these known measures as particular instances of monoidal width. We have previously captured~\cite{2022monoidalwidth} tree width, path width and branch width by instantiating monoidal width and two variants in a category of cospans of graphs.

\medskip
In this paper we focus on rank width~\cite{oum2006rank-width} -- a relatively recent development that has attracted significant attention in the graph theory community. A rank decomposition is a recipe for decomposing a graph into its single-vertex subgraphs by cutting along its edges. The cost of a cut is the rank of the adjacency matrix that represents it, as shown in \Cref{fig:ex-rank-cut}. A useful intuition for rank width is that it is a kind of \virgolette{Kolmogorov complexity} for graphs.
For example, although the family of cliques has unbounded tree width, the connectivity of cliques is quite simple to describe: and, in fact,  all cliques have rank width $1$.
\begin{figure}[h!]
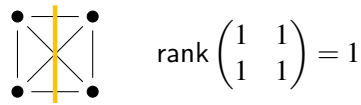

  \[\cutrankExFig{}\]
  \caption{A cut and its matrix in a rank decomposition.}\label{fig:ex-rank-cut}
\end{figure}

To capture rank width as an instance of monoidal width, rather than taking cospans, we work in a different monoidal category of graphs. First introduced in~\cite{chantawibul2015compositionalgraphtheory}, it was recently used~\cite{networkGamesCSL} as a syntax for network games. This approach to computing with ``open graphs'' is more linear algebraic, building modularly on the theory of bialgebra, well known to be closely related to matrix algebra~\cite{zanasi2015thesis}. Indeed, the connectivity of graphs is handled with adjacency matrices, and boundary connections are matrices.

%

\paragraph*{Related work.}
This manuscript, although self-contained, complements our previous work~\cite{2022monoidalwidth}, where we considered tree width, path width and branch width as instances of monoidal width.

Previous syntactical approaches to graph widths are the work of Pudl{\'a}k, R{\"o}dl and Savick{\`y}~\cite{pudlak1988graph-complexity} and the work of Bauderon and Courcelle~\cite{bauderon1987graph}.
Their works consider different notions of graph decompositions, which lead to different notions of graph complexity.
In particular, in~\cite{bauderon1987graph}, the cost of a decomposition is measured by counting \emph{shared names}, which is clearly closely related to penalising sequential composition as in monoidal width.
Nevertheless, these approaches are specific to particular, concrete notions of graphs, whereas our work concerns the more general algebraic framework of monoidal categories.

Recent abstract approaches focus on other graph widths.
The work of Blume et.\ al.~\cite{blume2011treewidth}, characterises tree and path decompositions in terms of colimits.
Abramsky et.\ al.~\cite{feder1998computational} give a coalgebraic characterization of  tree width of relational structures (and graphs in particular).
Bumpus and Kocsis~\cite{bumpus2021spined} also generalise tree width to the categorical setting, although their approach is far removed from ours.

\paragraph*{Synopsis.}
Monoidal width is recalled in \Cref{sec:monoidal-width}.
In \Cref{sec:mwd-matrices}, we study the monoidal width of matrices.
\Cref{sec:rank-width} recalls rank width and gives an equivalent recursive definition of it that will be useful as an intermediate step towards our main result, which is presented in \Cref{sec:mwd-rwd}. 

\paragraph*{Preliminaries.}
We use string diagrams~\cite{joyal1991geometry,selinger2010survey}:
sequential and parallel compositions of \(f\) and \(g\) are drawn as in \Cref{fig:string-diagrams}, left and middle, respectively.
Much of the bureaucracy, e.g. the interchange law \((f \dcomp g) \tensor (f' \dcomp g') = (f \tensor f') \dcomp (g \tensor g')\), disappears (\Cref{fig:string-diagrams}, right).
\begin{figure}[h!]
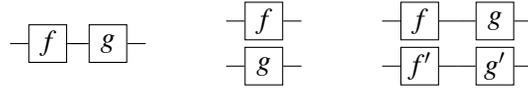

  \centering
  \sequentialFig{} \qquad \parallelFig{} \qquad \interchangelawFig{}
  \caption{String diagrammatic notation.}\label{fig:string-diagrams}
\end{figure}
\emph{Props}~\cite{MacLane1965,Lack2004a} are important examples of monoidal categories.
They are symmetric strict monoidal, with natural numbers as objects, and addition as monoidal product on objects. Roughly speaking, morphisms can be thought of as processes, and the objects (natural numbers) keep track of the number of inputs or outputs of a process.

\section{Monoidal width}\label{sec:monoidal-width}
This section recalls the concept of monoidal width from~\cite{2022monoidalwidth}.
Monoidal width records the cost of the most efficient way one can decompose a morphism into its atomic components, thus capturing---roughly speaking---its intrinsic structural complexity.
A decomposition is a binary tree whose internal nodes are labelled with compositions or monoidal products, and whose leaves are labelled with atomic morphisms.
\begin{definition}[Monoidal decomposition~\cite{2022monoidalwidth}]\label{defn:monoidalDecomposition}
  Let \(\cat{C}\) be a monoidal category and \(\decgenerators\) be a subset of its morphisms referred to as \emph{atomic}.
  The set $\decset{f}$ of \emph{monoidal decompositions} of \(f \colon A \to B\) in $\cat{C}$ is defined:
  \begin{align*}
  \decset{f} \quad \Coloneqq \quad & \leafgenerator{f} &\text{if } f \in \decgenerators \\
                      \mid \quad & \nodegenerator{d_{1}}{\tensor}{d_{2}} &\text{if } d_1\in \decset{f_1},\,d_2\in \decset{f_2} \text{ and } f = f_1\tensor f_2 \\
                      \mid \quad & \nodegenerator{d_{1}}{\dcomp_{X}}{d_{2}} &\text{if }d_1\in \decset{f_1 \colon A\to X},\,d_2\in \decset{f_2 \colon X\to B} \text{ and }f = f_1\dcomp f_2
  \end{align*}
\end{definition}
The cost of a decomposition depends on the operations and atoms present:
each operation and each atomic morphism is associated with a cost, which we call weight.
Roughly speaking, sequential composition is priced according to the size of the object the composition occurs over, while monoidal products are free.
Finally, the weight of an atom is the application-specific cost of computing its semantics.
\begin{definition}[Weight function~\cite{2022monoidalwidth}]
  Let \(\cat{C}\) be a monoidal category and let \(\decgenerators\) be a set of atoms for \(\cat{C}\).
  A weight function for \((\cat{C},\decgenerators)\) is a function \(\nodeweight \colon \decgenerators \union \monoidaloperations{\cat{C}} \to \naturals\) such that
  \begin{itemize}
    \item \(\nodeweight(X \tensor Y) = \nodeweight(X) + \nodeweight(Y)\),
    \item \(\nodeweight(\tensor) = 0\).
  \end{itemize}
\end{definition}

\begin{example}\label{ex:mon-dec}
  Let \(\cp_{1} \colon 1 \to 2\) and \(\cocp_{1} \colon 2 \to 1\) be the diagonal and codiagonal morphisms in a cartesian and cocartesian prop\footnote{In a \emph{cartesian prop} the $\tensor$ satisfies the universal property of  products. Dually, in a \emph{cocartesian prop}, the \(\tensor\) satisfies the universal property of the coproduct.} s.t.\ $\nodeweight(\cp_{1})=\nodeweight(\cocp_{1})=2$.
  The following figure represents the monoidal decomposition of \(\cp \dcomp (\cp \tensor \id{}) \dcomp (\cocp \tensor \id{}) \dcomp \cocp\) given by
  \( \nodegenerator{\cp}{\dcomp_{2}}{\nodegenerator{\nodegenerator{\nodegenerator{\cp}{\dcomp_{2}}{\cocp}}{\tensor}{\id{}}}{\dcomp_{2}}{\cocp}} \).
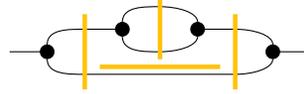
  \begin{center}
    \monoidaldecThreeExFig{}
  \end{center}
\end{example}

The width of a decomposition is the cost of the most expensive node in the decomposition tree.
\begin{definition}[Width of a monoidal decomposition~\cite{2022monoidalwidth}]\label{defn:decompositionWidth}
  Let \(\nodeweight\) be a weight function for \((\cat{C},\decgenerators)\).
  Let $f$ be in $\cat{C}$ and $d\in \decset{f}$.
  The width of $d$ is defined recursively as follows:
  \begin{align*}\label{eq:def-width-mon-dec}
    \decwidth(d) \defn\ &  \nodeweight(f) & \text{if }d=\leafgenerator{f} \\
                        &  \max \{\decwidth(d_1),\decwidth(d_2)\} & \text{if }d= \nodegenerator{d_{1}}{\tensor}{d_{2}}\\
                        &  \max \{\decwidth(d_1),\,\nodeweight(X),\,\decwidth(d_2)\} & \text{if } d= \nodegenerator{d_{1}}{\dcomp_{X}}{d_{2}}
  \end{align*}
\end{definition}
As sketched in~\Cref{ex:mon-dec}, decompositions can be seen as labelled trees $(S,\mu)$ where $S$ is a tree  and $\mu:\vertices(S)\to \decgenerators \union \monoidaloperations{\cat{C}}$ is a labelling function. With this we can restate the width as:
\[
\decwidth(d)=\decwidth(S,\mlabelling) \defn \max_{v \in \vertices(S)} \nodeweight(\mlabelling(v))
\]
which may be familiar to those aquainted with graph widths.

\smallskip
Monoidal width is simply the width of the cheapest decomposition.
\begin{definition}[Monoidal width~\cite{2022monoidalwidth}]\label{defn:monoidalWidth}
  Let \(\nodeweight\) be a weight function for \((\cat{C},\decgenerators)\)
  and $f$ be in $\cat{C}$. Then the \emph{monoidal width} of $f$ is
  \( \mwd(f) \defn \min_{d\in \decset{f}} \decwidth(d)\).
\end{definition}

\begin{example}\label{ex:mwd-number-like-morphisms}
  With the data of~\Cref{ex:mon-dec}, define a family of morphisms \(n \colon 1 \to 1\) inductively:\\
  \begin{minipage}{0.49\textwidth}
    \begin{itemize}
      \item \(1 \defn \id{1}\);
      \item \(2 \defn \cp \dcomp_{2} \cocp\);
      \item \(n+1 \defn \cp \dcomp_{2} (n \tensor 1) \dcomp_{2} \cocp\) for \(n \geq 2\).
    \end{itemize}
  \end{minipage}
  \begin{minipage}{0.49\textwidth}
    \begin{center}
      \mondecnumbersFig{}
    \end{center}
  \end{minipage}\\
  Each \(n\) has a monoidal decomposition of width \(n\): the root node is the composition along the \(n\) wires in the middle.
  However,  \(\mwd(n)=2\) for any \(n\), with an optimal  decomposition shown above.
\end{example}


\subsection{The width of copying}

Before we begin with the original technical contributions of this paper in Section~\ref{sec:mwd-matrices}, we need to recall a technical result from~\cite{2022monoidalwidth} about decomposing copy morphisms.
%
We consider symmetric monoidal categories equipped with such morphisms and show that copying \(n\) wires costs at most \(n+1\).
\begin{definition}[Copying]
  Let \(\cat{X}\) be a symmetric monoidal category with symmetries given by \(\swap{X,Y}\).
  We say that \(\cat{X}\) has \emph{coherent copying} if there is a class of objects $\mathcal{C}_\cat{X}\subseteq \obj{\cat{X}}$,
  satifying $X,Y\in \mathcal{C}_\cat{X}$ iff $X\tensor Y \in \mathcal{C}_{X}$, such that
  every \(X\) in \(\mathcal{C}_\cat{X}\)
   is endowed with a morphism \(\cp_{X} \colon X \to X \tensor X\).
   Moreover, \(\cp_{X \tensor Y} = (\cp_X \tensor \cp_Y) \dcomp (\id{X} \tensor \swap{X,Y} \tensor \id{Y})\) for every \(X,Y \in \mathcal{C}_\cat{X}\).
\end{definition}
An example is any cartesian prop with \(\cp_{n} \colon n \to n + n\) given by the cartesian structure.
We take \(\cp_{X}\), the symmetries \(\swap{X,Y}\) and the identities \(\id{X}\) as atomic for all objects \(X, Y\), i.e. the set of atomic morphisms is \(\decgenerators = \{\cp_{X},\, \swap{X,Y},\, \id{X} \ :\  X, Y \in \mathcal{C}_\cat{X}\}\).
The weight function is \(\nodeweight(\cp_{X}) \defn 2 \cdot \nodeweight(X)\), \(\nodeweight(\swap{X,Y}) \defn \nodeweight(X) + \nodeweight(Y)\) and \(\nodeweight(\id{X}) \defn \nodeweight(X)\).
In a prop, we take \(\nodeweight(n) \defn n\).
Note that $\nodeweight(\cp_{X\tensor Y}) = 2 \cdot \nodeweight(X\tensor Y) =
2\cdot (\nodeweight(X) + \nodeweight(Y))$, but utilising coherence we can do better, as illustrated below.

\begin{example}\label{ex:mwd-copy}
  Let \(\cat{C}\) be a prop with coherent copying and consider \(\cp_{n} \colon n \to 2n\).
  Let \(\gamma_{n,m} \defn (\cp_{n} \tensor \id{m}) \dcomp (\id{n} \tensor \swap{n,m}) \colon n + m \to n + m + n\).
  We can decompose \(\gamma_{n,m}\) in terms of \(\gamma_{n-1,m+1}\) (in the dashed box), \(\cp_{1}\) and \(\swap{1,1}\) by cutting along at most \(n+1+m\) wires:
  \begin{center}
    \mwdCopyExFig{}
  \end{center}
  This allows us to decompose \(\cp_{n} = \gamma_{n,0}\) cutting along at most \(n+1\) wires. In particular, $\mwd(\cp_n) \leq n+1$.
\end{example}
The following result is a technical generalisation of the argument presented in \Cref{ex:mwd-copy}.
\begin{lemma}[\cite{2022monoidalwidth}]\label{lemma:mwd-copy}
  Let \(\cat{X}\) be a symmetric monoidal category with coherent copying.
  Suppose that \(\decgenerators\) contains \(\cp_X\) for \(X \in \mathcal{C}_{\cat{X}}\), and \(\swap{X,Y}\) and \(\id{X}\) for \(X \in \obj{\cat{X}}\). Let $\overline{X}\defn X_1 \tensor \cdots \tensor X_n$,
  \(f \colon Y \tensor \overline{X} \tensor Z \to W\) and let \(d \in \decset{f}\).
  Let \(\gamma(f) \defn (\id{Y} \tensor \cp_{\overline{X}} \tensor \id{Z}) \dcomp (\id{Y \tensor \overline{X}} \tensor \swap{\overline{X}, Z}) \dcomp (f \tensor \id{\overline{X}})\).
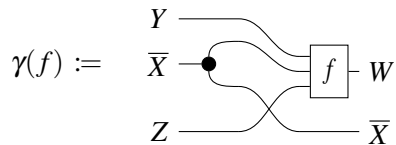
  \begin{center}
    \(\gamma(f) \,\defn\quad\)\lemmamwdcopyStateFig{}
  \end{center}
  There is a decomposition \(\copyMdec(d)\) of \(\gamma(f)\) of bounded width: 
  \[\decwidth(\copyMdec(d)) \leq \max \{\decwidth(d), \nodeweight(Y) + \nodeweight(Z) + (n+1) \cdot \max_{i = 1,\ldots,n} \nodeweight(X_i)\}.\]
\end{lemma}

\section{Monoidal width in matrices}\label{sec:mwd-matrices}
\begin{figure*}[h!]
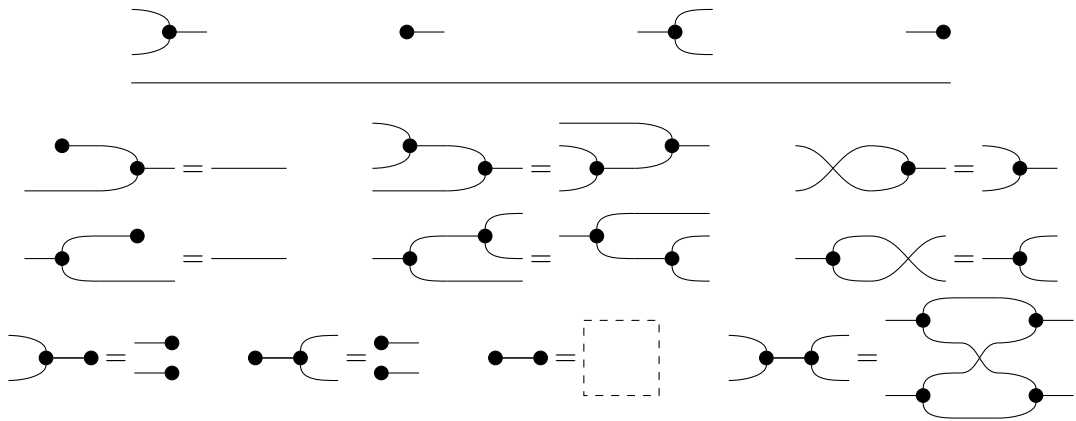

  \centering
  \bialgebraFig{}
  \caption{Bialgebra axioms}\label{fig:bialgebra}
\end{figure*}

Given the ubiquity of matrix algebra, matrices are an obvious case study.
\Cref{th:mwd-matrices} shows that the monoidal width of a matrix is, up to \(1\), the maximum of the ranks of its blocks.

Consider the monoidal category \(\Mat_{\naturals}\) of matrices with entries in the natural numbers.
The objects are natural numbers and morphisms from \(n\) to \(m\) are \(m\) by \(n\) matrices.
Composition is the usual product of matrices and the monoidal product is the biproduct: \(A \tensor B \defn \begin{psmallmatrix}A & \zeromat \\ \zeromat & B \end{psmallmatrix}\).
Let us examine matrix decompositions enabled by this algebra.
A matrix \(A\) can be written as a monoidal product \(A = A_{1} \tensor A_{2}\) iff the matrix has blocks \(A_{1}\) and \(A_{2}\), i.e. \(A = \begin{psmallmatrix} A_{1} & \zeromat \\ \zeromat & A_{2}\end{psmallmatrix}\).
On the other hand, a composition is related to the rank.
\begin{lemma}[\cite{piziak1999fullRank}]\label{lemma:min-cut-is-rank}
  Let \(A \colon n \to m\) in $\Mat_{\naturals}$.
  Then \(\min\{k \in \naturals : A = B \dcomp_k C\} = \rank(A)\).
\end{lemma}

\medskip
We first introduce a convenient syntax for matrices.

\begin{proposition}[\cite{zanasi2015thesis}]
  The category \(\Mat_{\naturals}\) is isomorphic to the prop \(\Bialg\), generated by \(\cp \colon 1 \to 2\), \(\delete \colon 1 \to 0\), \(\add \colon 2 \to 1\) and \(\zero \colon 0 \to 1\) and quotiented by bialgebra axioms (\Cref{fig:bialgebra}).
\end{proposition}

For the uninitiated reader, let us briefly explain this correspondence.
Every morphism \(f \colon n \to m\) in \(\Bialg\) corresponds to a matrix \(A = \BialgToMat(f) \in \Mat_{\naturals}(m,n)\): we can read the \((i,j)\)-entry of \(A\) off the diagram of \(f\) by counting the number of  paths from the \(j\)th input to the \(i\)th output.

\begin{example}
  The matrix \(\begin{psmallmatrix} 1 & 0 \\ 1 & 2 \\ 0 & 0 \end{psmallmatrix} \in \Mat_{\naturals}(3,2)\) corresponds to
  \begin{center}
    \matrixExFig{}
  \end{center}
\end{example}


\begin{definition}
  The atomic morphisms $\decgenerators$ are the generators of \(\Bialg\), with the symmetry and identity on \(1\):  \(\decgenerators = \{\cp,\delete,\add,\zero,\swap{},\id{1}\}\).
  The weight  \(\nodeweight \colon \decgenerators \union \monoidaloperations{\Bialg} \to \naturals\) has \(\nodeweight(\dcompnode{n}) \defn n\), for any \(n \in \naturals\), and \(\nodeweight(g) \defn \max\{m,n\}\), for \(g \colon n \to m \in \decgenerators\).
\end{definition}

\subsection{Monoidal width in \(\Bialg\)}

The characterisation of the rank of a matrix in \Cref{lemma:min-cut-is-rank} hints at some relationship between the monoidal width of a matrix and its rank.
In fact, we have \Cref{prop:matrix-mwd-non-tensor-decomposables}, which bounds the monoidal width of a matrix with its rank.
In order to prove this result, we first need to bound the monoidal width of a matrix with its domain and codomain, which is done in \Cref{prop:mwd-lessthan-domain-codomain}.

\begin{proposition}\label{prop:mwd-lessthan-domain-codomain}
  Let \(\cat{P}\) be a cartesian and cocartesian prop.
  Suppose that \(\id{1}, \cp_1, \add_1, \delete_1, \zero_1 \in \decgenerators\)
  and \(\nodeweight(\id{1}) \leq 1\), \(\nodeweight(\cp_1) \leq 2\), \(\nodeweight(\add_1) \leq 2\), \(\nodeweight(\delete_1) \leq 1\) and \(\nodeweight(\zero_1) \leq 1\).
  Suppose that, for every \(g \colon 1 \to 1\), \(\mwd(g) \leq 2\).
  Let \(f \colon n \to m\) be a morphism in \(\cat{P}\).
  Then \(\mwd(f) \leq \min\{m,n\}+1\).
\end{proposition}
\begin{proof}[Proof sketch]
  The proof proceeds by induction on \(\max\{m,n\}\).
  The base cases are easily checked.
  The inductive step relies on the fact that, applying~\Cref{lemma:mwd-copy}, if \(n < m\), we can decompose \(f\) as shown below by cutting at most \(n+1\) wires or, if \(m <n\), in the symmetric way by cutting at most \(m+1\) wires.
  \[\mwdMatlessthanBoundariesOneShort{}\]
\end{proof}

We can apply the former result to \(\Bialg\) and obtain \Cref{prop:matrix-mwd-non-tensor-decomposables} because the width of \(1 \times 1\) matrices, which are numbers, is at most \(2\).
This follows from the reasoning in~\Cref{ex:mwd-number-like-morphisms} as we can write every natural number \(k \colon 1 \to 1\) as the following composition:
\begin{center}
  \mondecnumbersFig{}
\end{center}
\begin{lemma}\label{lemma:mwd-numbers}
  Let \(k \colon 1 \to 1\) in \(\Bialg\).
  Then, \(\mwd(k) \leq 2\).
\end{lemma}

\begin{proposition}\label{prop:matrix-mwd-non-tensor-decomposables}
  Let \(f \colon n \to m\) in \(\Bialg\).
  Then, \(\mwd f \leq \rank(\BialgToMat f) +1\).
  Moreover, if \(f\) is not \(\tensor\)-decomposable, i.e. there are no \(f_1,f_2\) both distinct from \(f\) s.t. \(f = f_1 \tensor f_2\), then \(\rank(\BialgToMat f) \leq \mwd f\).
\end{proposition}
\begin{proof}[Proof sketch]
  This result follows from~\Cref{lemma:min-cut-is-rank} and~\Cref{prop:mwd-lessthan-domain-codomain}, which we can apply thanks to \Cref{lemma:mwd-numbers}.
\end{proof}

The bounds given by \Cref{prop:matrix-mwd-non-tensor-decomposables} can be improven when we have a \(\tensor\)-decomposition of a matrix, i.e. we can write \(f = f_{1} \tensor \dots \tensor f_{k}\), to obtain \Cref{prop:matrices-better-tensors-first}.
The latter relies on \Cref{lemma:mwd-after-discard}, which shows that discarding inputs or outputs cannot increase the monoidal width of a morphism in \(\Bialg\).

\begin{lemma}\label{lemma:mwd-after-discard}
  Let \(f \colon n \to m\) in \(\Bialg\) and \(d \in \decset{f}\).
  Let \(f_{D} \defn f \dcomp (\id{m-k} \tensor \delete_k)\) and \(f_{Z} \defn (\id{n-k'} \tensor \zero_{k'}) \dcomp f\), where \(\delete_{k} \colon k \to 0\) is the discard morphism with \(k \leq m\) and \(\zero_{k'} \colon 0 \to k\) is the zero morphism with \(k' \leq n\).
  \[\mwdafterdiscardStateFig{}\]
  Then there are \(\deleteMdec(d) \in \decset{f_{D}}\) and \(\codeleteMdec(d) \in \decset{f_{Z}}\) such that \(\decwidth(\deleteMdec(d)) \leq \decwidth(d)\) and \(\decwidth(\codeleteMdec(d)) \leq \decwidth(d)\).
\end{lemma}
\begin{proof}[Proof sketch]
  By induction.
  The base cases are easy.
  If \(f = f_1 \dcomp f_2\), use the inductive hypothesis on \(f_{2}\).
  \[\mwdafterdiscardProofFigTwo{}\]
  The \(f = f_{1} \tensor f_{2}\) case is similar.
\end{proof}

\begin{proposition}\label{prop:matrices-better-tensors-first}
  Let \(f \colon n \to m\) in \(\Bialg\) and \(d' = \nodegenerator{d'_1}{\dcomp_k}{d'_2} \in \decset{f}\).
  Suppose there are \(f_1\) and \(f_2\) such that \(f = f_1 \tensor f_2\).
  Then, there is \(d = \nodegenerator{d_1}{\tensor}{d_2} \in \decset{f}\) such that \(\decwidth(d) \leq \decwidth(d')\).
\end{proposition}
\begin{proof}[Proof sketch]
  By~\Cref{lemma:min-cut-is-rank}, \(\rank(\BialgToMat f_{1}) + \rank(\BialgToMat f_{2}) = \rank(\BialgToMat(f_{1} \tensor f_{2})) \leq k\) and, by~\Cref{prop:matrix-mwd-non-tensor-decomposables}, there is a monoidal decomposition \(d_{i}\) of \(f_{i}\) such that \(\decwidth(d_{i}) \leq \rank(\BialgToMat f_{i}) + 1\).
  Then, \(\decwidth(d) \defn \decwidth(\nodegenerator{d_{1}}{\tensor}{d_{2}}) \leq \max \{\rank(\BialgToMat f_{1}), \rank(\BialgToMat f_{2})\} + 1 \leq \rank(\BialgToMat f_{1}) + \rank(\BialgToMat f_{2})\) whenever \(\rank(\BialgToMat f_{1}), \rank(\BialgToMat f_{2}) > 0\).
  We apply \Cref{lemma:mwd-after-discard} to obtain the same result if \(\rank(\BialgToMat f_1) = 0\) or \(\rank(\BialgToMat f_2) = 0\).
\end{proof}

We summarise \Cref{prop:matrices-better-tensors-first} and \Cref{prop:matrix-mwd-non-tensor-decomposables} in \Cref{cor:mwd-matrices-blocks}.

\begin{corollary}\label{cor:mwd-matrices-blocks}
  Let \(f = f_{1} \tensor \dots \tensor f_{k}\) in \(\Bialg\).
  Then, \(\mwd(f) \leq \max_{i = 1,\dots,k} \rank(\BialgToMat(f_{i})) + 1\).
  Moreover, if \(f_i\) are not \(\tensor\)-decomposable, then \(\max_{i = 1,\dots,k} \rank(\BialgToMat(f_{i})) \leq \mwd f\).
\end{corollary}
\begin{proof}
  By \Cref{prop:matrices-better-tensors-first} there is a decomposition of \(f\) of the form \(d = \nodegenerator{d_1}{\tensor}{\cdots \nodegenerator{d_{k-1}}{\tensor}{d_k}}\), where we can choose \(d_i\) to be a minimal decomposition of \(f_i\).
  Then, \(\mwd(f) \leq \decwidth(d) = \max_{i = 1,\ldots,k} \decwidth(d_i)\).
  By \Cref{prop:matrix-mwd-non-tensor-decomposables}, \(\decwidth(d_i) \leq r_i +1\).
  Then, \(\mwd(f) \leq \max\{r_1,\ldots,r_k\}+1\).
  Moreover, if \(f_i\) are not \(\tensor\)-decomposable, \Cref{prop:matrix-mwd-non-tensor-decomposables} gives also a lower bound on their monoidal width: \(\rank(\BialgToMat(f_{i})) \leq \mwd f_i\); and we obtain that \(\max_{i = 1,\dots,k} \rank(\BialgToMat(f_{i})) \leq \mwd f\).
\end{proof}

The results so far show a way to construct efficient decompositions given a \(\tensor\)-decomposition of the matrix.
However, we do not know whether \(\tensor\)-decompositions are unique.
\Cref{prop:mon-cat-unique-tensor-decomposition} shows that every morphism in \(\Bialg\) has a unique \(\tensor\)-decomposition.

\begin{proposition}\label{prop:mon-cat-unique-tensor-decomposition}
  Let \(\cat{C}\) be a monoidal category whose monoidal unit \(0\) is both initial and terminal, and whose objects are a unique factorization monoid.
  Let \(f\) be a morphism in \(\cat{C}\).
  Then \(f\) has a unique \(\tensor\)-decomposition.
\end{proposition}

Our main result in this section follows from~\Cref{cor:mwd-matrices-blocks} and~\Cref{prop:mon-cat-unique-tensor-decomposition}, which can be applied to \(\Bialg\) because \(0\) is both terminal and initial, and the objects, being a free monoid, are a unique factorization monoid.

\begin{theorem}\label{th:mwd-matrices}
  Let \(f = f_1 \tensor \ldots \tensor f_k\) be a morphism in \(\Bialg\) and its unique \(\tensor\)-decomposition given by \Cref{prop:mon-cat-unique-tensor-decomposition}, with \(r_i = \rank(\BialgToMat(f_i))\).
  Then \(\max\{r_1,\ldots,r_k\} \leq \mwd(f) \leq \max\{r_1,\ldots,r_k\}+1\).
\end{theorem}
\begin{proof}
  This result is obtained by applying \Cref{cor:mwd-matrices-blocks} to the \(\tensor\)-decomposition given by \Cref{prop:mon-cat-unique-tensor-decomposition}, which can be applied because, in \(\Bialg\), \(0\) is both terminal and initial, and the objects, being a free monoid, are a unique factorization monoid.
\end{proof}
Note that the identity matrix 
has monoidal width \(1\)
and twice the identity matrix 
has monoidal width \(2\), attaining both the upper and lower bounds for the monoidal width of a matrix.

\section{Graphs and rank width}\label{sec:rank-width}
Here we recall rank width~\cite{oum2006rank-width} for undirected graphs.
\begin{definition}
  An \emph{undirected graph} \(G = \mathgraph[,\edgeendsfun]{E}{V}\) is given by a set of edges \(E\), a set of vertices \(V\) and a function \(\edgeendsfun \colon E \to \parti_{\leq 2}(V)\) that gives the endpoints of each edge.
  We consider graphs \emph{up to isomorphism}, or \emph{abstract graphs}, thus the set of vertices can be fully characterised by its cardinality.
  An abstract graph can be equivalently given by an adjacency matrix \(\adjeqclass{G}\), where \(G \in \MatN(n,n)\) and \(n\) is the number of vertices.
  The equivalence class of adjacency matrices is defined by the equivalence relation
  \[G \adjeqrel H \quad \text{ iff } \quad G + \transpose{G} = H + \transpose{H}.\]
\end{definition}
We will refer to abstract undirected graphs as simply graphs.
\begin{definition}
  A \emph{path} in a graph \(G\) is a sequence of edges \((e_{1}, \dots, e_{k})\) together with a sequence of distinct vertices \((v_{1},\dots,v_{k+1})\) of \(G\) such that, for every \(i = 1,\dots,k\), \(\edgeends{e_{i}} = \{v_{i}, v_{i+1}\}\).
  A \emph{tree} is a graph such that there is a unique path between any two of its vertices.
  Two vertices \(v\) and \(w\) in a graph \(G\) are \emph{neighbours} if \(G\) has an edge between them.
  The \emph{leaves} of a tree are those vertices with at most one neighbour.
  A \emph{subcubic tree} is a tree where each vertex has between one and three neighbours.
\end{definition}

A rank decomposition for a graph \(G\) is a tree whose leaves are labelled with the vertices of \(G\).
\begin{definition}[~\cite{oum2006rank-width}]
  A \emph{rank decomposition} \((Y,r)\) of a graph \(G\) is given by a subcubic tree \(Y\) together with a bijection \(r \colon \leaves(Y) \to \vertices(G)\).
\end{definition}
Each edge \(b\) in the tree \(Y\) determines a splitting of the graph: it determines a two partition of the leaves of \(Y\), which, through \(r\), determines a two partition \(\{A_b,B_b\}\) of the vertices of \(G\).
This corresponds to a splitting of the graph \(G\) into two subgraphs \(G_{1}\) and \(G_{2}\).
Intuitively, the order of an edge \(b\) is the amount of information required to recover \(G\) by joining \(G_{1}\) and \(G_{2}\).
Given the partition \(\{A_b,B_b\}\) of the vertices of \(G\), we can record the edges in \(G\) beween \(A_{b}\) and \(B_{b}\) in a matrix \(X_{b}\).
This means that, if \(v_{i} \in A_{b}\) and \(v_{j} \in B_{b}\), the entry \((i,j)\) of the matrix \(X_{b}\) is the number of edges between \(v_{i}\) and \(v_{j}\).

\begin{definition}[Order of an edge]\label{def:edge-order}
  Let \((Y,r)\) be a rank decomposition of a graph \(G\).
  Let \(b\) be an edge of \(Y\).
  The order of \(b\) is the rank of the matrix associated to it: \(\edgeorder(b) \defn \rank(X_{b})\).
\end{definition}
Note that the order of the two sets in the partition does not matter as the rank is invariant to transposition.
The width of a rank decomposition is the maximum order of the edges of the tree and the rank width of a graph is the width of its cheapest decomposition.
\begin{definition}[Rank width]
  Given a rank decomposition \((Y,r)\) of a graph \(G\), define its width as \(\decwidth(Y,r) \defn \max_{b \in \edges(Y)} \edgeorder(b)\).
  The \emph{rank width} of \(G\) is given by the min-max formula:
  \[\rankwidth(G) \defn \min_{(Y,r)} \decwidth(Y,r).\]
\end{definition}

\subsection{Graphs with dangling edges}
As intermediate step between rank decompositions and monoidal decompositions, we introduce recursive rank decompositions of \emph{graphs with dangling edges} and we prove that they give a notion of width that is equivalent to rank width.
Similar recursive characterisations were done for tree decompositions in~\cite{arnborg1993algebraic} and for path and branch decompositions in~\cite{2022monoidalwidth}.
We first need a notion of graph that is equipped with some \virgolette{open} edges along which it can be glued with other graphs.
\begin{definition}\label{def:dangling-graph}
  A \emph{graph with dangling edges} \(\Gamma = \danglinggraph{G}{B}\) is given by an adjacency matrix \(G \in \MatN(k,k)\) that records the connectivity of the graph and a matrix \(B \in \MatN(k,n)\) that records the ``dangling edges'' connected to $n$ boundary ports.
  We will sometimes write \(G \in \adjacency(\Gamma)\) and \(B = \boundary(\Gamma)\).
\end{definition}

\begin{example}\label{ex:graph-dangling-edges-glueing}
  Two graphs with the same ports, as illustrated below, can be \virgolette{glued} together: 
  \begin{center}
    \danglinggraphIdeaExFigOne{} glued with \danglinggraphIdeaExFigTwo{} gives \danglinggraphIdeaExFigThree{}
  \end{center}
\end{example}
Decompositions are elements of a tree data type, with nodes carrying subgraphs \(\Gamma'\) of the ambient graph \(\Gamma\).
In the following \(\Gamma'\) ranges over the non-empty subgraphs of \(\Gamma\):
\(T_{\Gamma} \ \Coloneqq \ \leafgenerator{\Gamma'} \ \mid \ \nodegenerator{T_{\Gamma}}{\Gamma'}{T_{\Gamma}}\).
Given $T\in T_\Gamma$, the label function $\labelling$ takes a decomposition and returns the graph with dangling edges at the root: $\labelling\nodegenerator{T_1}{\Gamma}{T_2} \defn \Gamma$ and \(\labelling\leafgenerator{\Gamma} \defn \Gamma\).
\begin{definition}[Recursive rank decomposition]\label{def:rec-rank-dec}
  Let \(\Gamma = \danglinggraph{G}{B}\) be a graph with dangling edges, where \(G \in  \MatN(k,k)\) and \(B \in \MatN(k,n)\).
  A recursive rank decomposition of \(\Gamma\) is \(T \in T_{\Gamma}\) where either:
  \(\Gamma\) has at most one vertex and \(T = \leafgenerator{\Gamma}\);
  or \(T = \nodegenerator{T_1}{\Gamma}{T_2}\) and \(T_i \in T_{\Gamma_i}\) are recursive rank decompositions of subgraphs \(\Gamma_i = \danglinggraph{G_i}{B_i}\) of \(\Gamma\) such that:
  \begin{itemize}
    \item The vertices are partitioned in two, \(\adjeqclass{G} = \adjeqclass{\begin{psmallmatrix} G_1 & C \\ \zeromat & G_2 \end{psmallmatrix}}\);
    \item The dangling edges are those to the original boundary and to the other subgraph, \(B_1 = (A_1 \mid C)\) and \(B_2 = (A_2 \mid \transpose{C})\), where \(B = \begin{psmallmatrix} A_1 \\ A_2 \end{psmallmatrix}\).
  \end{itemize}
\end{definition}
As with before, the \emph{recursive} rank width of a graph is the width of its cheapest decomposition.
\begin{definition}
  Let \(T\) be a recursive rank decomposition of \(\Gamma = \danglinggraph{G}{B}\).
  Define the width of \(T\) recursively:
  if \(T = \leafgenerator{\Gamma}\), \(\decwidth(T) \defn \rank(B)\),
  and, if \(T = \nodegenerator{T_{1}}{\Gamma}{T_{2}}\), \(\decwidth(T) \defn \max\{\decwidth(T_1), \decwidth(T_2), \rank(B)\}\)
  Expanding this expression, we obtain
  \(\decwidth(T) = \max_{T' \text{ subtree of } T} \rank(\boundary(\labelling(T')))\).
  The \emph{recursive rank width} of \(\Gamma\) is defined by the min-max formula
  \(\rrankwidth(\Gamma) \defn \min_{T} \decwidth(T)\).
\end{definition}
We show that recursive rank width is the same as rank width, up to the rank of the boundary of the graph.
\begin{proposition}\label{prop:recursive-rwd-upper-bound}
  Let \(\Gamma = \danglinggraph{G}{B}\) be a graph with dangling edges and \((Y,r)\) be a rank decomposition of \(G\).
  Then, there is a recursive rank decomposition \(\toRecursiveDec(Y,r)\) of \(\Gamma\) s.t. \(\decwidth(\toRecursiveDec(Y,r)) \leq \decwidth(Y,r) + \rank(B)\).
\end{proposition}
Before proving the lower bound for recursive rank width, we need a technical lemma that relates the width of a graph with that of its subgraphs.
\begin{lemma}\label{lemma:closed-expr-boundary}
  Let \(T\) be a recursive rank decomposition of \(\Gamma = \danglinggraph{G}{B}\).
  Let \(T'\) be a subtree of \(T\) and \(\Gamma' \defn \labelling(T')\) with \(\Gamma'= \danglinggraph{G'}{B'}\).
  The adjacency matrix of \(\Gamma\) can be written as \(\adjeqclass{G} = \adjeqclass{\begin{psmallmatrix} G_L & C_L & C \\ \zeromat & G' & C_R \\ \zeromat & \zeromat & G_R \end{psmallmatrix}}\) and its boundary as \(B = \begin{psmallmatrix} A_L \\ A' \\ A_R \end{psmallmatrix}\).
  Then, \(\rank(B') = \rank(A' \mid \transpose{C_L} \mid C_R)\).
\end{lemma}

\begin{proposition}\label{prop:recursive-rwd-lower-bound}
  Let \(T\) be a recursive rank decomposition of \(\Gamma = \danglinggraph{G}{B}\) with \(G \in \MatN(k,k)\) and \(B \in \MatN(k,n)\).
  Then, there is a rank decomposition \(\fromRecursiveDec(T)\) of \(G\) such that \(\decwidth(\fromRecursiveDec(T)) \leq \decwidth(T)\).
\end{proposition}

From \Cref{prop:recursive-rwd-lower-bound} and \Cref{prop:recursive-rwd-upper-bound} we conclude the following result.

\begin{theorem}\label{th:recursive-rwd}
  Let \(\Gamma = \danglinggraph{G}{B}\).
  Then,
  \(\rankwidth(G) \leq \rrankwidth(\Gamma) \leq \rankwidth(G) + \rank(B)\).
\end{theorem}

\section{Monoidal width and rank width}\label{sec:mwd-rwd}
This section contains our main results.
We prove that monoidal width in the prop of graphs \(\propGraph\)~\cite{chantawibul2015compositionalgraphtheory} corresponds to rank width, up to a constant multiplicative factor of \(2\).

We start by introducing the algebra of graphs with boundaries
and its diagrammatic syntax~\cite{networkGamesCSL}.
A graph with boundaries is a graph together with two matrices \(L\) and \(R\) that record the connectivity of the vertices with the left and right boundary, a matrix \(P\) that records the passing wires from the left boundary to the right one and a matrix \(F\) that records the wires from the right boundary to itself.
\begin{definition}[\cite{networkGamesCSL}]
  A \emph{graph with boundaries} \(g \colon n \to m\) is defined as \(g = \fullboundariesgraph{G}{L}{R}{P}{F}\), where \(\adjeqclass{G}\) is the adjacency matrix of a graph on \(k\) vertices, with \(G \in \MatN(k,k)\); \(L \in \MatN(k,n)\),  \(R \in \MatN(k,m)\), \(P \in \MatN(m,n)\) and \(F \in \MatN(m,m)\) recording connectivity information as explained above.
  Graphs with boundaries are taken up to an equivalence making the order of the vertices immaterial.
  Let \(g, g' \colon n \to m\) on \(k\) vertices, with \(g = \fullboundariesgraph{G}{L}{R}{P}{F}\) and \(g' = \fullboundariesgraph{G'}{L'}{R'}{P}{F}\).
  The graphs \(g\) and \(g'\) are considered equal iff
  there is a permutation matrix \(\sigma \in \MatN(k,k)\) such that
  \(g' = \fullboundariesgraph{\sigma G \transpose{\sigma}}{\sigma L}{\sigma R}{P}{F}\).
\end{definition}
Graphs with boundaries can be composed sequentially and in parallel~\cite{networkGamesCSL}, forming a symmetric monoidal category \(\catGraph\).
The prop \(\propGraph\) provides a convenient syntax for graphs with boundaries.
It is obtained by adding a cup and a vertex generators to the prop of matrices \(\Bialg\) (\Cref{fig:bialgebra}).
\begin{definition}[\cite{chantawibul2015compositionalgraphtheory}]
  The prop of graphs \(\propGraph\) is obtained by adding to \(\Bialg\) the generators \(\cup \colon 0 \to 2\) and \(\vertex \colon 1 \to 0\) with the equations below.
  \[\propGraphEqFig{}\]
\end{definition}
These equations mean, in particular, that the cup transposes matrices
(\Cref{fig:adding-cup}, left) and that we can express the equivalence relation of adjacency matrices: \(G \adjeqrel H\) iff \(G + \transpose{G} = H + \transpose{H}\) (\Cref{fig:adding-cup}, right).
\begin{figure}[h!]
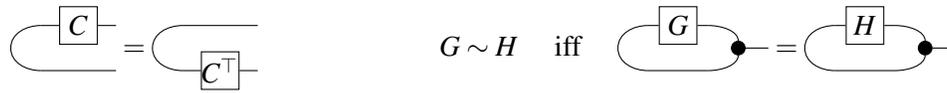

  \centering
  \(\cupTransposeFig{}\) \qquad\qquad\qquad\(\adjeqrelFig{}\)
  \caption{Adding the cup.}\label{fig:adding-cup}
\end{figure}

\begin{proposition}[\cite{networkGamesCSL}, Theorem 23]\label{prop:normal-form-prop-graphs}
  The prop of graphs \(\propGraph\) is isomorphic to the prop \(\catGraph\).
\end{proposition}
\Cref{prop:normal-form-prop-graphs} means that the morphisms in \(\propGraph\) can be written in the following normal form
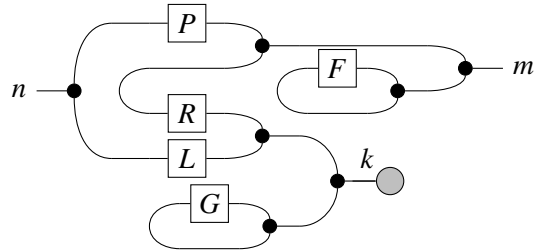
\[\morphismpropgraphExFig{}.\]

The prop \(\propGraph\) is more expressive than graphs with dangling edges (\Cref{def:dangling-graph}): its morphisms can have edges between the boundaries as well.
In fact, graphs with dangling edges can be seen as morphisms \(n \to 0\) in \(\propGraph\).

\begin{example}
  A graph with dangling edges \(\Gamma = \danglinggraph{G}{B}\) can be represented as a morphism in \(\propGraph\)
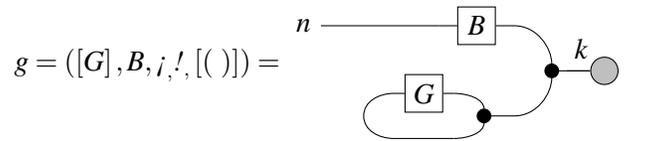
  \[g = \fullboundariesgraph{G}{B}{\initmap}{\finmap}{\emptymat} = \danglinggraphExFig{}\quad.\]
  We can now formalise the intuition of glueing graphs with dangling edges as explained in \Cref{ex:graph-dangling-edges-glueing}.
  The two graphs there correspond to \(g_{1}\) and \(g_{2}\) below left and middle.
  Their glueing is obtained by precomposing their monoidal product with a cup, i.e. \ \(\cup_{2} \dcomp (g_{1} \tensor g_{2})\), as shown below right.
  \[\danglinggraphGlueingExFig\]
\end{example}

\subsection{Rank width in open graphs}
The technical content of our main result (\Cref{th:mwd-rwd}) is split in two: an upper  and a lower bound.

As in the prop of matrices \(\Bialg\), the cost of composing along \(n\) wires is \(n\).
All morphisms in \(\propGraph\) are chosen as atomic.
One could restrict this to those with at most one vertex without affecting the results.
\begin{definition}
  Let the set of \emph{atomic morphisms} \(\decgenerators\) be the set of all the morphisms of \(\propGraph\).
  The \emph{weight function} \(\nodeweight \colon \decgenerators \union \monoidaloperations{\propGraph} \to \naturals\) is defined, on objects \(n\), as \(\nodeweight(\dcompnode{n}) \defn n\); and, on morphisms \(g \in \decgenerators\), as \(\nodeweight(g) \defn k\), where \(k\) is the number of vertices of \(g\).
\end{definition}
Note that the monoidal width of \(g\) is bounded by the number of its vertices.

The upper bound (\Cref{prop:mwd-rwd-upper-bound}) is established by associating to each recursive rank decomposition a suitable monoidal decomposition.
This mapping is defined inductively, given the inductive nature of both these structures.
Given a recursive rank decomposition of a graph \(\Gamma\), we can construct a decomposition of its corresponding morphism \(g\) as shown by the first equality in \Cref{fig:mwd-rwd-upper-bound}.
\begin{figure}[h!]
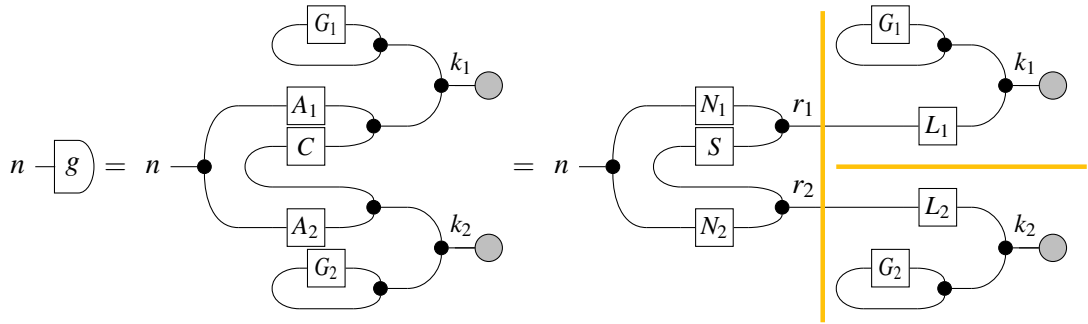

  \[\mwdrwdupperproofFigOneCuts{}\]
  \caption{First step of a monoidal decomposition given by a recursive rank decomposition}\label{fig:mwd-rwd-upper-bound}
\end{figure}
However, this decomposition is not optimal as it cuts along the number of vertices \(k_{1} + k_{2}\).
But we can do better thanks to \Cref{lemma:cut-along-ranks}, which shows that we can cut along the ranks, \(r_1 = \rank(A_1 \mid C)\) and \(r_2 = \rank(A_2 \mid \transpose{C})\), of the boundaries of the induced subgraphs to obtain the second equality in \Cref{fig:mwd-rwd-upper-bound}.
\[\lemmacutalongranksIdeaFig{}\]
\begin{lemma}\label{lemma:cut-along-ranks}
  Let \(A_i \in \MatN(k_i,n)\), for \(i = 1,2\), and \(C \in \MatN(k_1,k_2)\).
  Then, there are rank decompositions of \((A_1 \mid C)\) and \((A_2 \mid \transpose{C})\) of the form \((A_1 \mid C) = L_1 \cdot (N_1 \mid S \cdot \transpose{L_2})\), and \((A_2 \mid \transpose{C}) = L_2 \cdot (N_2 \mid \transpose{S} \cdot \transpose{L_1})\).
\end{lemma}
Once we have performed the cuts in \Cref{fig:mwd-rwd-upper-bound} on the right, we have changed the boundaries of the induced subgraphs.
This means that we cannot apply the inductive hypothesis right away, but we need to transform first the recursive rank decompositions of the old subgraphs into decompositions of the new ones, as shown in \Cref{lemma:rank-on-boundary}.
More explicitly, when \(M\) has full rank, if we have a recursive rank decomposition of \(\Gamma = \danglinggraph{G}{B' \cdot M}\), which corresponds to \(g\) below left, we can obtain one of \(\Gamma' = \danglinggraph{G}{B'}\), which corresponds to \(g'\) below right, of the same width.
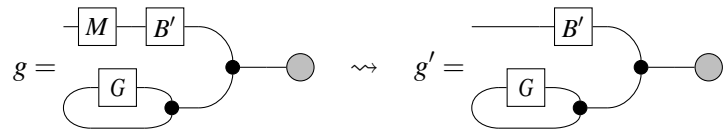
\[\lemmarankonboundaryIdeaFig{}\]
\begin{lemma}\label{lemma:rank-on-boundary}
  Let \(T\) be a recursive rank decomposition of \(\Gamma = \danglinggraph{G}{B}\) and \(B = B' \cdot M\), with \(M\) that has full rank.
  Then, there is a recursive rank decomposition \(T'\) of \(\Gamma' = \danglinggraph{G}{B'}\) such that \(\decwidth(T) = \decwidth(T')\) and such that \(T\) and \(T'\) have the same underlying tree structure.
\end{lemma}
With the above ingredients, we can show that rank width bounds monoidal width from above.
\begin{proposition}\label{prop:mwd-rwd-upper-bound}
  Let \(\Gamma = \danglinggraph{G}{B}\) be a graph with dangling edges and \(g \colon n \to 0\) be the morphism in \(\propGraph\) corresponding to \(\Gamma\).
  Let \(T\) be a recursive rank decomposition of \(\Gamma\).
  Then, there is a monoidal decomposition \(\rTomdec(T)\) of \(g\) such that \(\decwidth(\rTomdec(T)) \leq 2 \cdot \decwidth(T)\).
\end{proposition}
\begin{proof}[Proof sketch]
  The proof proceeds by induction on \(T\).
  The base cases are easily checked and the inductive step relies on the decomposition of \(g\) in \Cref{fig:mwd-rwd-upper-bound}, which we can write thanks to \Cref{lemma:cut-along-ranks}.
  Applying the inductive hypothesis and \Cref{lemma:rank-on-boundary}, the width of this decomposition can be bounded by \(\max \{r_1 + r_2, 2 \cdot \decwidth(T_1),  2 \cdot \decwidth(T_2)\} \leq 2 \cdot \decwidth(T)\), where \(T = \nodegenerator{T_1}{\Gamma}{T_2}\).
\end{proof}
Proving the lower bound is similarly involved and follows a similar proof structure.
From a monoidal decomposition we construct inductively a recursive rank decomposition of bounded width.
The inductive step relative to composition nodes is the most involved and needs two additional lemmas, which allow us to transform recursive rank decompositions of the induced subgraphs into ones of two subgraphs that satisfy the conditions of \Cref{def:rec-rank-dec}.

Applying the inductive hypothesis gives us a recursive rank decomposition of \(\Gamma = \danglinggraph{G}{(L \mid R)}\), which is associated to \(g\) below left, and we need to construct one of \(\Gamma' \defn \danglinggraph{G + L \cdot F \cdot \transpose{L}}{(L \mid R + L \cdot (F + \transpose{F}) \cdot \transpose{P})}\), which is associated to \(f \dcomp g\) below right, of at most the same width.
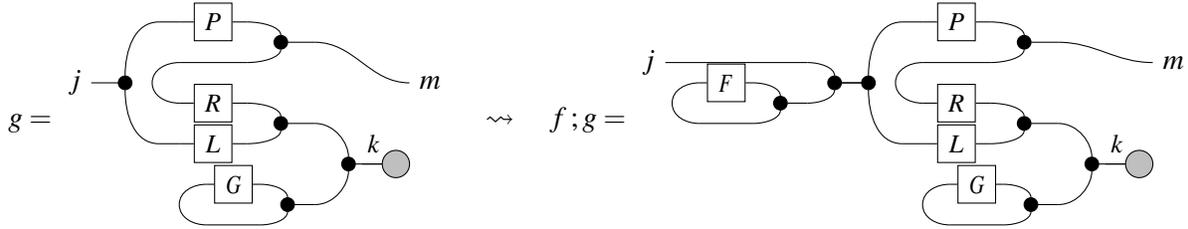
\[\lemmawirestofutureIdeaFig{}\]
\begin{lemma}\label{lemma:wires-to-future}
  Let \(T\) be a recursive rank decomposition of \(\Gamma = \danglinggraph{G}{(L \mid R)}\).
  Let \(F \in \MatN(j,j)\), \(P \in \MatN(m,j)\) and define \(\Gamma' \defn \danglinggraph{G + L \cdot F \cdot \transpose{L}}{(L \mid R + L \cdot (F + \transpose{F}) \cdot \transpose{P})}\).
  Then, there is a recursive rank decomposition \(T'\) of \(\Gamma'\) of bounded width: \(\decwidth(T') \leq \decwidth(T)\).
\end{lemma}
In order to obtain the subgraphs of the desired shape we need to add some extra connections to the boundaries.
We have a recursive rank decomposition of \(\Gamma = \danglinggraph{G}{B}\), which corresponds to \(g\) below left, and we need one of \(\Gamma' = \danglinggraph{G}{B \cdot M}\), which corresponds to \(g'\) below right, of at most the same width.
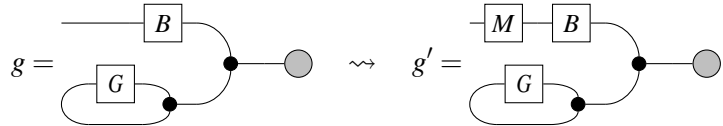
\[\lemmaaddtoboundaryIdeaFig{}\]
The following result and its proof are very similar to \Cref{lemma:rank-on-boundary}.
\begin{lemma}\label{lemma:paste-to-boundary}
  Let \(T\) be a recursive rank decomposition of \(\Gamma = \danglinggraph{G}{B}\) and let \(B' = B \cdot M\).
  Then, there is a recursive rank decomposition \(T'\) of \(\Gamma' = \danglinggraph{G}{B'}\) such that \(\decwidth(T') \leq \decwidth(T)\) and such that \(T\) and \(T'\) have the same underlying tree structure.
  Moreover, if \(M\) has full rank, then \(\decwidth(T') = \decwidth(T)\).
\end{lemma}

\begin{proposition}\label{prop:mwd-rwd-lower-bound}
  Let \(g = \fullboundariesgraph{G}{L}{R}{P}{F}\) in \(\propGraph\) and \(d \in \decset{g}\).
  Let \(\Gamma = \danglinggraph{G}{(L \mid R)}\).
  Then, there is a recursive rank decomposition \(\mTordec(d)\) of \(\Gamma\) s.t.
  \(\decwidth(\mTordec(d)) \leq 2 \cdot \max \{\decwidth(d), \rank(L), \rank(R)\}\).
\end{proposition}
\begin{proof}[Proof sketch]
  The proof proceeds by induction on \(d\).
  The base case is easily checked, while the inductive steps are a bit more involved.
  If \(d = \nodegenerator{d_1}{\dcomp_j}{d_2}\), then there are \(g_i = \fullboundariesgraph{G_i}{L_i}{R_i}{P_i}{F_i}\) such that \(g = g_1 \dcomp g_2\) and we can write \(g\) as follows.
  \[\mwdrankwidthlowerproofFigFuture{}\]
  In order to build a recursive rank decomposition of \(\Gamma\), we need recursive rank decompositions of \(\overline{\Gamma}_i = \danglinggraph{\overline{G}_i}{\overline{B}_i}\), but we can obtain recursive rank decompositions of \(\Gamma_i = \danglinggraph{G_i}{(L_i \mid R_i)}\) by applying only induction.
  Thanks to \Cref{lemma:wires-to-future}, we obtain a recursive rank decomposition of \(\Gamma'_2 = \danglinggraph{G_2 + L_2 \cdot F_1 \cdot \transpose{L_2}}{(L_2 \mid R_2 + L_2 \cdot (F_1 + \transpose{F_1}) \cdot \transpose{P_2})}\).
  Lastly, we apply \Cref{lemma:paste-to-boundary} to get recursive rank decompositions \(T_i\) of \(\overline{\Gamma}_i\).
  Thanks to these, we can bound the width of \(T \defn \nodegenerator{T_1}{\Gamma}{T_2}\):
  \[\decwidth(T) \leq 2 \cdot \max \{\decwidth(d_1), \decwidth(d_2), j, \rank(L), \rank(R)\} \codefn 2 \cdot \max \{\decwidth(d), \rank(L), \rank(R)\}.\]
  If \(d = \nodegenerator{d_1}{\tensor}{d_2}\), we proceed similarly.
\end{proof}

From \Cref{prop:mwd-rwd-upper-bound}, \Cref{prop:mwd-rwd-lower-bound} and \Cref{th:recursive-rwd}, we obtain our main result.

\begin{theorem}\label{th:mwd-rwd}
  Let \(G\) be a graph and let \(g = \fullboundariesgraph{G}{\initmap}{\initmap}{\emptymat}{\emptymat}\) be the corresponding morphism of \(\propGraph\).
  Then, \(\frac{1}{2} \cdot \rankwidth(G) \leq \mwd(g) \leq 2 \cdot \rankwidth(G)\).
\end{theorem}

\section{Conclusions and future work}
We have shown that monoidal width, in a suitable category of graphs composable along \virgolette{open} edges, yields rank width; a well-known measure from the graph theory literature.

Our goal with this line of research is to develop a generic, abstract ``decomposition theory''.
We will study other graph widths like clique width~\cite{courcelle2000upper-clique} and twin width~\cite{bonnet2020twin}, as well as go beyond graphs: e.g.\ by focussing on tree width for hypergraphs and relational structures~\cite{abramsky2017pebbling}, branch width for matroids and widths for directed graphs.
A part of \virgolette{decomposition theory} means going beyond width as a mere number -- in fact we believe that in each case the identification of a suitable monoidal category as an \emph{algebra} of open graph structures is itself a worthwhile contribution. Indeed, having such an algebra means that a decomposition, rather than an ad hoc concept-specific construction, becomes more of a mathematical object in its own right.  Such compositional algebras will add to the quiver of compositional structures of applied category theory; for example serving as syntax for more sophisticated applications~\cite{networkGamesCSL}.


\paragraph*{Acknowledgements.}
Elena Di Lavore and {Pawe\l} Sobociński were supported by the European Union through the ESF funded Estonian IT Academy research measure (2014-2020.4.05.19-0001). This work was also supported by the Estonian Research Council grant PRG1210.
\bibliographystyle{eptcs}
\bibliography{mwd-biblio.bib}
\end{document}